\documentclass[12pt]{amsart}

\usepackage{amssymb,latexsym}

\usepackage{color}

\usepackage{enumerate}

\usepackage[T1]{fontenc}

 \usepackage[french,english]{babel}

 \usepackage{xcolor}
 
\usepackage{hyperref}


\newtheorem{thm}{Theorem}[section]

\newtheorem{lem}[thm]{Lemma}
\newtheorem{pro}[thm]{Proposition}
\theoremstyle{definition}

\newtheorem{rem}[thm]{Remark}

\numberwithin{equation}{section}

\renewcommand{\le}{\leqslant}
\renewcommand{\leq}{\leqslant}

\renewcommand{\geq}{\geqslant}

\newcommand{\R}{\mathbb{R}}
\newcommand{\C}{\mathbb{C}}

\newcommand{\re}{\textup{Re}}

\newcommand{\mc}{\mathcal}

\newcommand{\Ns}{N_0^{\textup{s}}}

\newcommand{\newabstract}[1]{%
  \par\bigskip
  \csname otherlanguage*\endcsname{#1}%
  \csname captions#1\endcsname
  \item[\hskip\labelsep\scshape\abstractname.]
}

\begin{document}

\baselineskip=17pt

\title[$2/3$ of the zeros of the zeta function are simple and on the critical line
]{A new proof that more than $2/3$ of the zeros of the Riemann zeta function are simple and on the critical line
}

\author{Youness Lamzouri}

\address{
Universit\'e de Lorraine, CNRS, IECL, 
F-54000 Nancy, France}

\email{youness.lamzouri@univ-lorraine.fr}

\date{\today}

\begin{abstract} 
We obtain a new, conceptually simpler, unconditional proof that more than $67.25\%$ of the non-trivial zeros of the Riemann zeta function are simple and on the critical line, and that at least $83.62\%$ of the non-trivial zeros are distinct. Our approach also yields two new unconditional estimates on simple zeros and zeros on the critical line. More precisely, we prove that the proportion of zeros that are simple or lie on the critical line (or both) is at least $88.76\%$, and that the average of the proportions of simple zeros and of zeros on the critical line is at least $83.62\%$. A proof of the bounds for simple zeros on the critical line and for distinct zeros was very recently produced by an internal research version of Claude developed by Anthropic and subsequently verified by two mathematicians at Anthropic, Levent Alp\"oge and Ralph Furman, whereas our two additional estimates are neither stated nor proved in the Claude paper.
The argument produced by Claude is technically intricate, and its main mechanism is not immediately transparent. It combines several ingredients from linear algebra, including a finite-dimensional matrix representation of Weil's Hermitian form and a rank--trace inequality for Hermitian matrices, with a second moment calculation over the zeros using the explicit formula.
 Our new approach proceeds by replacing the entire finite-dimensional matrix framework by a single Hilbert space inequality, which allows for a direct application of Montgomery's theorem on the pair correlation of zeros of the zeta function, in the unconditional form obtained by Baluyot, Goldston, Suriajaya and
Turnage-Butterbaugh.  

\end{abstract}


\maketitle


\section{Introduction} 
The Riemann hypothesis (RH) is one of the central open problems in mathematics and one of the seven Millennium Prize problems. This conjecture, which was formulated by Riemann in 1859, asserts that all the non-trivial zeros\footnote{That is, the zeros lying in the critical strip $0\leq\re(s)\leq 1$.} of the Riemann zeta function lie on the critical line $\re(s)=1/2$. Its truth would have tremendous applications in number theory, most notably for our understanding of the distribution of prime numbers.

Let $\rho=\beta+i\gamma$ denote a non-trivial zero of the Riemann zeta function, and let $m_\rho$ be its multiplicity. We write
$$
N(T):=\sum_{\substack{\rho\\0<\gamma\leq T}}1,
$$
where the zeros are counted with multiplicity. The classical Riemann--von Mangoldt formula gives
$
N(T)\sim \frac{T}{2\pi}\log T,
$
as $T\to\infty$.
In the absence of a proof of the Riemann hypothesis, a fundamental problem has been to determine what proportion of the zeros can be shown to lie on the critical line. The first progress in this direction goes back to Hardy \cite{Har14}, who proved in 1914 that infinitely many zeros lie on the line $\re(s)=1/2$. Hardy and Littlewood \cite{HL21} subsequently showed that the number of such zeros up to height $T$ is $\gg T$. A fundamental breakthrough was made by Selberg \cite{Sel42}, who pioneered the use of mollification in this problem and proved in 1942 that a positive proportion of the zeros lie on the critical line. In 1974, Levinson \cite{Lev74} introduced a new mollifier method and used it to prove that more than one third of the zeros lie on the line $\re(s)=1/2$. Conrey \cite{Con89} later developed Levinson's method further and proved in 1989 that more than $2/5$ of the zeros are simple and lie on the critical line. More recently, Pratt, Robles, Zaharescu and Zeindler \cite{PRZZ20} obtained the slightly better bounds
$$
\liminf_{T\to\infty}\frac{N_0(T)}{N(T)}\geq 0.417293
\qquad\text{and}\qquad
\liminf_{T\to\infty}\frac{\Ns(T)}{N(T)}\geq 0.407511,
$$
where
$$
N_0(T):=\sum_{\substack{\rho\\0<\gamma\leq T\\\beta=1/2}}1
\qquad\textup{and}\qquad
\Ns(T):=\left|\left\{\rho:0<\gamma\leq T,\ \beta=\frac12,\ \text{and }m_\rho=1\right\}\right|,
$$
and where the zeros in the definition of $N_0(T)$ are counted with multiplicity. We also let
$$
N_s(T):=\sum_{\substack{\rho\\0<\gamma\leq T\\m_\rho=1}}1,
\quad
N_{s\cup 0}(T):=
\sum_{\substack{\rho\\0<\gamma\leq T\\m_\rho=1\ \textup{or}\ \beta=1/2}}1,
\quad \textup{ and } \quad
N_d(T):=\left|\left\{\rho:0<\gamma\leq T\right\}\right|.
$$ Thus $N_s(T)$ denotes the number of simple zeros, $N_d(T)$ the number of distinct zeros, and $N_{s\cup 0}(T)$ the number of zeros that are simple or lie on the critical line, counted with multiplicity, all up to height $T$.

In 1973, Montgomery \cite{Mo73} established his pair correlation theorem for the zeros of the zeta function and used it to prove, assuming RH, that at least $2/3$ of the zeros are simple. Montgomery and Taylor (see \cite{Mo75}) subsequently obtained the larger proportion 
\begin{equation}\label{Eq:DefinitionC_0}
C_0:=\frac32-\frac{1}{\sqrt{2}}\cot\left(\frac{1}{\sqrt{2}}\right)
=0.6725007\ldots,
\end{equation}
and this was improved to $0.6792$ by Chirre, Gon\c{c}alves, and de Laat \cite{CGL20}. 
By a different method, Conrey, Ghosh and Gonek \cite{CGG98} proved, assuming RH and the Generalized Lindel\"of Hypothesis (GLH), that at least
$
\frac{19}{27}=0.7037\ldots
$ of the zeta zeros are simple, and that at least $0.8456$ of the zeros are distinct. Bui and Heath-Brown \cite{BHB13} subsequently removed GLH from this result and obtained the proportion $19/27$ assuming only RH.
Unconditionally, Farmer \cite{Far95} proved that more than $63.95\%$ of the zeros are distinct. This was improved by Ki and Lee \cite[Theorem 5]{KL12}, who proved that more than $70\%$ of the zeros are distinct.

More recently, substantial progress has been made towards removing the Riemann hypothesis from Montgomery's pair correlation argument. Aryan \cite{Ary22} obtained an unconditional form of the pair correlation formula for the Fej\'er kernel and showed that Montgomery's $2/3$ bound for simple zeros follows from a zero-density hypothesis. Baluyot, Goldston, Suriajaya and Turnage-Butterbaugh \cite{BGST24} subsequently established an unconditional version of Montgomery's pair correlation theorem and, under hypotheses weaker than RH, proved that at least $61.7\%$ of the zeros are simple. In further work \cite{BGST25}, they showed for the first time that the pair correlation method can also give information on the horizontal distribution of the zeros. 
More precisely, they proved that if for all large $T$, all the zeros with $T<\gamma\leq 2T$ satisfy
$|\beta-\frac12|<\frac{b}{2\log T},$
then at least $2/3$ of the zeros are simple and on the critical line when $b=0.3185$, while for $b=0.001$ the corresponding proportion is at least $0.6725$.
 Goldston and Suriajaya \cite{GS25,GS26} subsequently clarified the mechanism behind these results. 
Moreover, extending earlier work of Gallagher and Mueller \cite{GM78}, Goldston, Lee, Schettler and Suriajaya \cite{GLSS25} showed that Montgomery's full Pair Correlation Conjecture, without assuming RH, implies that asymptotically $100\%$ of the zeros are simple and on the critical line.

The main purpose of this paper is to give a new and conceptually simpler approach to the breakthrough result recently announced by an internal research version of Claude, developed by Anthropic, and subsequently verified by Alp\"oge and Furman \cite{AlFu26}. In addition to recovering their bounds, our approach yields two further unconditional estimates concerning simple zeros and zeros on the critical line.

\begin{thm}\label{Thm:Main}
 We have
$$
 \liminf_{T\to \infty} \frac{\Ns(T)}{N(T)}
 \geq C_0=0.67250\ldots,
 \quad \textup{ and } \quad 
 \liminf_{T\to \infty} \frac{N_d(T)}{N(T)}
 \geq C_1:=\frac{C_0+1}{2}=0.83625\ldots.
$$
Moreover, we have
$$
  \liminf_{T\to\infty}
 \frac{N_{s\cup 0}(T)}{N(T)}
 \geq
 C_2
 =
 0.88762\ldots
 \ \textup{ and } \
 \liminf_{T\to\infty}
 \frac{N_{s}(T)+N_0(T)}{2N(T)}
 \geq C_1=0.83625\ldots 
.
$$
Here $C_0$ is defined in \eqref{Eq:DefinitionC_0} and $C_2:=(1+2\sqrt{2}+2C_0)/(3+2\sqrt{2}).$
Thus, more than $67.25\%$ of the non-trivial zeros of the zeta function are simple and on the critical line, and more than $88.76\%$ of the zeros are either simple or lie on the critical line (or both). In addition, more than $83.62\%$ of the zeros are distinct, and the average of the proportions of simple zeros and of zeros on the critical line is at least $83.62\%$.
\end{thm}
\begin{rem}
The two additional estimates express a form of ``rigidity'' between the simplicity and the horizontal location of zeros. The first shows that, asymptotically, the proportion of zeros that are simultaneously off the critical line and of multiplicity at least two is at most $1-C_2<0.1124$. The second shows in particular that if either $N_s(T)/N(T)$ or $N_0(T)/N(T)$ tends to the lower bound $C_0$, then the other must tend to $1$.
Note that the last estimate implies that $\max(N_0(T), N_s(T))\geq (C_1+o(1))N(T)$. Thus, although our argument does not determine which of the two quantities is larger, it shows that the larger of the proportions of zeros on the critical line and of simple zeros is asymptotically at least
$83.62\%$. Related estimates are obtained in Theorem~3(ii) and (iii) of Baluyot, Goldston, Suriajaya and Turnage-Butterbaugh \cite{BGST25} through their notion of \emph{horizontal multiplicity}, under the hypothesis that, for all large $T$, all zeros with $T<\gamma\leq 2T$ lie in a narrow box of width $b/\log T$ centered on the critical line. In the case $b\to 0$, their Theorem~3(ii) yields the same lower bound
$
\frac{C_0+1}{2}=C_1=0.83625\ldots$
for the average of the proportions of simple zeros and of zeros on the critical line. Their Theorem~3(iii), following an observation of Soundararajan, gives under the same assumptions the lower bound
$
\frac{2+C_0}{3}=0.89083\ldots$
for the proportion of zeros that are simple or lie on the critical line (or both), whereas our approach yields the slightly smaller unconditional proportion $0.88762\ldots$.
\end{rem}
We briefly compare our proof with the argument discovered by Claude and verified by Alp\"oge and Furman \cite{AlFu26}. Their proof starts from Weil's Hermitian form and restricts it to a finite family of test functions, producing a real symmetric matrix of dimension asymptotic to $N(T)$. They then evaluate its trace and the sum of the squares of its entries (that is, the square of its Hilbert--Schmidt norm) and combine these estimates with a rank--trace inequality. To this end, they separate the part of the matrix coming from simple zeros on the critical line from the remaining part. The rank--trace inequality then gives a lower bound for the number of simple zeros on the critical line in terms of $N(T)$, the trace of the matrix, and the square of its Hilbert--Schmidt norm.

Our proof takes a different and more direct route. We avoid the finite-dimensional matrix construction altogether, and replace the linear algebra part of the argument by a single Hilbert space inequality. This reduces the problem directly to an estimate for a pair correlation sum over the zeros, 
to which we apply Montgomery's pair correlation theorem in the unconditional form obtained by Baluyot, Goldston, Suriajaya and Turnage-Butterbaugh (see Lemma 5 of \cite{BGST24}). The Hilbert space formulation makes the underlying structure more transparent and allows additional inequalities to be extracted directly.

Although the two proofs are quite different, they both ultimately reduce the relevant information to the estimation of a certain quadratic form over the zeros. In this sense, both may be viewed as variants of a second-moment argument, and the optimization leads in both cases to the same Montgomery--Taylor extremal problem (see Remark \ref{Rem:MT} below), which explains why the same constants appear in the two proofs.
\subsection*{Acknowledgments} The author is especially grateful to Brian Conrey, Daniel Goldston,  Steve Gonek, Andrew Granville, Micah Milinovich and Gérald Tenenbaum for carefully reading an earlier draft of the manuscript and for their valuable comments. He also thanks Kannan Soundararajan for helpful suggestions, and Ade Irma Suriajaya for helpful comments and discussions. The author is supported by a junior chair of the Institut Universitaire de France. The author is also grateful to Ken Ono and the Axiom Math team for their enthusiasm and for producing, on very short notice and following the completion of this work, the Lean formal certificates described in Appendix A.


\section{The key proposition}

Since non-trivial zeros of the Riemann zeta function are counted with multiplicity, we will be working with multisets instead of sets of zeros.  Montgomery's classical argument in \cite{Mo73}, which shows that at least $2/3$ of the zeros of the zeta function are simple assuming RH, relies on the following basic inequality
\begin{equation}\label{Eq:Inequality1Montgomery}
\sum_{\substack{\rho\\ 0<\gamma\leq T\\m_{\rho}=1}}1
\geq
\sum_{\substack{\rho\\ 0<\gamma\leq T}}(2-m_\rho)
=
2\sum_{\substack{\rho\\ 0<\gamma\leq T}}1
-
\sum_{\substack{\rho,\rho'\\ 0<\gamma, \gamma'\leq T\\\rho=\rho'}}1,
\end{equation}
which is true unconditionally. Therefore, one reduces the problem to obtaining an upper bound for the diagonal contribution on the right hand side of \eqref{Eq:Inequality1Montgomery}. If RH is assumed then one has 
\begin{equation}\label{Eq:InequalityMontgomery2}
\sum_{\substack{\rho,\rho'\\ 0<\gamma, \gamma'\leq T\\\rho=\rho'}}1\leq  \sum_{\substack{\rho,\rho'\\ 0<\gamma, \gamma'\leq T}} K\left(i(\rho-\rho')\frac{\log T}{2\pi}\right)\frac{4}{4-(\rho-\rho')^2},
\end{equation} 
for any kernel\footnote{Since in this case $i(\rho-\rho')= \gamma'-\gamma \in \mathbb{R}$.} $K$ such that $K(0)=1$ and $K(u)\geq 0$ for all $u\in \R$. Assuming RH, Montgomery then computes this pair correlation sum with $K$ being the Fej\'er kernel $K(u)=\left(\frac{\sin(\pi u)}{\pi u}\right)^2$. 
Recent attempts to use this approach unconditionally relied on some form of positivity for $K$ when evaluated at differences of non-trivial zeros that may lie off the critical line. In \cite{BGST25}, Baluyot, Goldston, Suriajaya and
Turnage-Butterbaugh achieved this using a kernel of Tsang whose real part is positive in a fixed horizontal strip. After rescaling by $\log T/(2\pi)$, this positivity applies only if one assumes that all the non-trivial zeros with imaginary part in $[T, 2T]$  lie in a very narrow strip of the form $|\beta-1/2|\leq c/\log T$. This restriction cannot be removed merely by searching for a better kernel. Indeed, for the above method to work without any information on the real parts of the zeros, one would require that the kernel $K$ satisfies
$
\re K(z)\geq0
$
for all $z\in \C$. However, no such nonconstant entire kernel exists.

Our key idea is that one can bypass the inequality \eqref{Eq:Inequality1Montgomery} and directly prove, for a large class of kernels $K$, that the number of simple real elements of a finite multiset $\mc{Z}$ which is invariant under complex conjugation is
$$
\geq 2\sum_{z\in\mc{Z}}1
-
\sum_{z,s\in\mc{Z}}K(z-s)^2,
$$
without requiring that the individual off-diagonal terms $K(z-s)^2$ be non-negative.

For a compactly supported function $f\in L^{1}(\mathbb R)$ we define its Fourier transform $\widehat{f}:\C\to \C$ by
\begin{equation}\label{hat-g}
\widehat{f}(\xi) = \int_{-\infty}^{\infty} f(u)e^{-2\pi i\xi u }\,du, \quad \text{ for all } \xi\in \C.
\end{equation}
Note that this is an entire function since $f$ is compactly supported. 
For an element $z$ of a finite multiset $\mc{Z}$, we denote by $m_z$ its multiplicity. We say that a finite multiset $\mc{Z}$ is \emph{invariant under complex conjugation} if for all $z\in \mc{Z}$ we have $\overline{z}\in \mc{Z}$ and $m_{\overline{z}}=m_z.$  Throughout, sums over a multiset are understood to count elements with their multiplicities. We prove the following key proposition.
\begin{pro}\label{Pro:FourierHilbert}
Let $\lambda>0$ be a real number and $\eta\in L^2(\R)$ be a real valued even function with $\textup{supp} (\eta) \subset (-\lambda, \lambda) $, such that $\widehat{\eta^2}(0)=1$. Let $\mc{Z}$ be a non-empty finite multiset of complex numbers which is invariant under complex
conjugation. For $z\in \mc{Z}$ we let $m_z$ be its multiplicity in $\mc{Z}$.  Then we have 
\begin{equation}\label{Eq:SimpleRealZeros}
 \sum_{\substack{z\in \mc{Z}\cap \R\\ m_z=1}}1\geq 2\sum_{z\in \mc{Z}} 1 -
 \sum_{z,s\in\mathcal Z}K(z-s)^2,  
\end{equation}
where $K$ is the kernel $K(\xi):= \widehat{\eta^2}(\xi).$ Moreover, the number of distinct elements of $\mc{Z}$ is 
\begin{equation}\label{Eq:DistinctZeros}
 \geq \frac{3}{2}\sum_{z\in \mc{Z}} 1 -
 \frac12\sum_{z,s\in\mathcal Z}K(z-s)^2.
\end{equation}
In addition, we have 
\begin{equation}\label{Eq:SumSimpleCritical}
\sum_{\substack{z\in \mc{Z}\\ m_z=1}}1 + \sum_{z\in \mc{Z}\cap \R}1 \geq 3\sum_{z\in \mc{Z}} 1-  \sum_{z,s\in\mathcal Z}K(z-s)^2.
\end{equation}
Finally, if we assume that $\sum_{z,s\in\mathcal Z}K(z-s)^2\leq A \sum_{z\in \mc{Z}} 1$, for some $1\leq A<2$, then we have 
\begin{equation}\label{Eq:SimpleOrCritical}
\sum_{\substack{z\in \mc{Z}\\ m_z=1 \text{  or }z\in \R}}1  \geq \frac{5+2\sqrt{2}-2A}{3+2\sqrt{2}}\sum_{z\in \mc{Z}} 1.
\end{equation}
\end{pro}
We observe that the unconditional version of the pair correlation formula (see Lemma \ref{Lem:BGST} below)  produces the weighted sum containing an extra factor $4/(4-(\rho-\rho')^2)$, rather than the unweighted quadratic form which appears in Proposition \ref{Pro:FourierHilbert}. An additional modification of the test function will be used in Section \ref{Sec:AnalyticPart} to remove this weight.
\begin{rem}\label{Rem:KernelChoice}
There are two reasons for working with $K^2$, instead of $K$, in Proposition \ref{Pro:FourierHilbert}.
 Indeed, since the number of simple real elements of $\mc Z$ is at most $\sum_{z\in\mc Z}1$, the inequality \eqref{Eq:SimpleRealZeros} implies that
$\sum_{z,s\in\mc Z}K(z-s)^2
\geq
\sum_{z\in\mc Z}1
$. This is substantially stronger than positive definiteness alone. The square  is also naturally adapted to the Hilbert space argument in the proof and to the use of Bessel's inequality. 
\end{rem}
\begin{proof}[Proof of Proposition \ref{Pro:FourierHilbert}]
For all $z\in \C$ we define the functions $f_z, g_z, h_z:\R\to \C$ by
\begin{equation}\label{Eq:DefinitionFz}
 f_z(u):=\eta(u) e^{-2\pi i u z}, \quad 
 g_{z}(u) := \frac{f_z(u)+f_{\overline{z}}(u)}{2}, \quad \text{ and } \quad h_{z}(u)= \frac{f_z(u)-f_{\overline{z}}(u)}{2i}.
\end{equation}
Then $g_{\overline{z}}= g_{z}$, $h_{\overline{z}}= -h_z$ and $f_{z}=g_z+ i h_z$.  
Moreover, we observe that
\begin{equation}\label{Eq:KernelFactorization}
K(z-\overline{s})=\widehat{\eta^2}(z-\overline{s})= \int_{-\lambda}^{\lambda}\eta(u)e^{-2\pi i z u}\eta(u)e^{2\pi i\overline{s}u}\,du= \int_{-\lambda}^{\lambda} f_z(u)\overline{f_s(u)}\,du.
\end{equation}
Hence for all $x\in \R$ and $z\in \C$ we have
\begin{equation}\label{Eq:IntegralFz}
\int_{-\lambda}^{\lambda} |f_x(u)|^2du=K(0)=1,
\end{equation}
and 
\begin{align}\label{Eq:IntegralGH}
1&=\int_{-\lambda}^{\lambda} f_z(u)\overline{f_{\overline{z}}(u)} du=\re \int_{-\lambda}^{\lambda}(g_z(u)+ i h_z(u))(\overline{g_z(u)}+ i \overline{h_z(u)})du\nonumber\\
&=\int_{-\lambda}^{\lambda} |g_z(u)|^2 du- \int_{-\lambda}^{\lambda} |h_z(u)|^2 du.
\end{align}
Now for all $(u, v) \in \R^2$ we define
$$
F(u,v):=\sum_{z\in\mathcal Z}f_z(u)f_z(v).
$$
Since $\mc{Z}$ is invariant under complex conjugation we have 
\begin{align}\label{Eq:SecondMoment}
\sum_{z, s\in \mc{Z}} K^2 (z-s)&=\sum_{z, s\in \mc{Z}} K^2 (z-\overline{s})= \sum_{z, s\in \mc{Z}}  \left(\int_{-\lambda}^{\lambda}  f_z(u) \overline{f_s(u)}  \,du \right)^2 \nonumber \\ &= \sum_{z, s\in \mc{Z}} \int_{-\lambda}^{\lambda} \int_{-\lambda}^{\lambda} f_z(u) f_z(v) \overline{f_s(u) f_s(v)} \,du \, dv
=\int_{-\lambda}^{\lambda} \int_{-\lambda}^{\lambda} |F(u, v)|^2 \,du \,dv.
\end{align}

Let $\mc{R}$ be the set of distinct real elements of $\mc{Z}$, and let $\mc{S}$ be the set of its distinct non-real elements. We further split the set $\mc{R}$ into two disjoint subsets $\mc{R}_1$ and $\mc{R}_2$, where $\mc{R}_1$ is the set of simple real elements of $\mc{Z}$, and $\mc{R}_2$ is the set of those elements which are real and have multiplicity $\geq 2$. Since $\mc{Z}$ is invariant under complex conjugation we can list the elements of these sets as 
$$
\mc{R}_1= \{x_1, \dots, x_n\}, \quad \mc{R}_2= \{x_{n+1}, \dots, x_{n+r}\}, \ \text{ and } \ \mc{S}=\{z_1, \overline{z}_1, \dots, z_k, \overline{z}_k\}.
$$
Grouping the non-real elements by pairs of complex conjugates we deduce that 
\begin{align}\label{Eq:IdentityFUV}
F(u, v) 
&= \sum_{\ell=1}^{n+ r} m_{x_\ell}f_{x_\ell}(u)f_{x_\ell}(v)+ \sum_{\ell=1}^k  m_{z_\ell} \left(f_{z_\ell}(u)f_{z_\ell}(v)+ f_{\overline{z_\ell}}(u)f_{\overline{z_\ell}}(v)\right)\nonumber\\
&= \sum_{\ell=1}^{n+ r} m_{x_\ell}f_{x_\ell}(u)f_{x_\ell}(v)+ 2\sum_{\ell=1}^k  m_{z_\ell} \left(g_{z_\ell}(u) g_{z_\ell}(v)- h_{z_\ell}(u) h_{z_\ell}(v)\right).
\end{align}
We now consider the following nested finite dimensional subspaces of the Hilbert space $L^2((-\lambda, \lambda), \C)$, endowed with its usual inner product $\langle f, g\rangle = \int_{-\lambda}^{\lambda} f(u)\overline{g(u)}du$:
$$
U=\textup{Span}\big(f_{x_{n+1}}, \dots, f_{x_{n+r}}, g_{z_1}, \dots, g_{z_k}\big), \quad V=\textup{Span}\big(f_{x_1}, \dots, f_{x_{n+r}}, g_{z_1}, \dots, g_{z_k}\big), 
$$ and 
$$ W=\textup{Span}\big(f_{x_1}, \dots, f_{x_{n+r}}, g_{z_1}, \dots, g_{z_k}, h_{z_1}, \dots, h_{z_k}\big).$$
We also put 
$$D_U:=\dim U, \quad D_V:= \dim V \ \text{ and } \ D_W:= \dim W. $$
 Note that for distinct complex numbers $z_1, \dots, z_{\ell}$, the functions $e^{z_1 u}, e^{z_2 u}, \dots, e^{z_\ell u}$ are linearly independent over $\C$. Therefore, since the set on which $\eta$ is non-zero has positive measure (because $\widehat{\eta^2}(0)=1$), it follows that $f_{x_1}, \ldots, f_{x_{n+r}}, g_{z_1}, \ldots, g_{z_{k}}, h_{z_1}, \ldots, h_{z_k}$ are linearly independent over $\C$. Hence, we get $D_U=r+k$, $D_V=n+r+k$ and $D_W=n+r+2k$.
One can also observe that all of the functions $f_{x_{\ell}}, g_{z_{\ell}}$ and $h_{z_{\ell}}$ satisfy the following property
\begin{equation}\label{Eq:PropertyOverline}
\overline{\Phi(u)}= \Phi(-u). 
\end{equation} 
Moreover, if $\Phi_1$ and $\Phi_2$ satisfy \eqref{Eq:PropertyOverline} then 
\begin{align}\label{Eq:RealIntegral}
\overline{\langle \Phi_1, \Phi_2\rangle} 
 &=\int_{-\lambda}^{\lambda} \overline{\Phi_1(u)}\Phi_2(u)\,du
 =\int_{-\lambda}^{\lambda} \Phi_1(-u)\overline{\Phi_2(-u)}\,du,\nonumber\\
 &=\int_{-\lambda}^{\lambda} \Phi_1(u)\overline{\Phi_2(u)}\,du
=\langle \Phi_1, \Phi_2\rangle 
\end{align}
by a simple change of variables.
We now apply the standard Gram-Schmidt process to construct an orthonormal basis $(\psi_1, \dots, \psi_{D_W})$ of $W$, such that  $(\psi_1, \dots, \psi_{D_U})$ is a basis for $U$ and $(\psi_1, \dots, \psi_{D_V})$ is a basis for $V$. Since the original spanning family satisfies \eqref{Eq:PropertyOverline}, the identity \eqref{Eq:RealIntegral} shows inductively that all the Gram–Schmidt coefficients are real. Hence every $\psi_j$ also satisfies \eqref{Eq:PropertyOverline}, and so all the inner product coefficients appearing below are real. Furthermore, the functions $\Psi_j:\R^2\to \mathbb{C}$, defined by $\Psi_j(u, v)=\psi_j(u)\psi_j(v)$, are orthonormal in $L^2\big((-\lambda, \lambda)^2, \C\big)$. Indeed, for $1\leq j,\ell\leq D_W$ we have
$$ \int_{-\lambda}^{\lambda}\int_{-\lambda}^{\lambda} \psi_j(u)\psi_j(v) \overline{\psi_{\ell}(u)\psi_{\ell}(v)}du dv= \bigg(\int_{-\lambda}^{\lambda}\psi_j(u)\overline{\psi_{\ell}(u)} du\bigg)^2.$$
 Hence by Bessel's inequality we get
\begin{equation}\label{Eq:OrthogonalInequality}
\int_{-\lambda}^{\lambda}\int_{-\lambda}^{\lambda} |F(u, v)|^2 du dv \geq \sum_{j=1}^{D_W}|\alpha_j|^2,
\end{equation}
where 
\begin{align}\label{Eq:DefinitionAlpha}
\alpha_j&:= \int_{-\lambda}^{\lambda}\int_{-\lambda}^{\lambda} F(u, v) \overline{\psi_j(u)\psi_j(v)}\,du \,dv \nonumber\\
&= \sum_{\ell=1}^{n+ r} m_{x_{\ell}}\left(\int_{-\lambda}^{\lambda} f_{x_{\ell}}(u)\overline{\psi_j(u)}du\right)^2\nonumber\\
& \qquad \qquad+ 2\sum_{\ell=1}^k  m_{z_{\ell}} \left( \left(\int_{-\lambda}^{\lambda} g_{z_\ell}(u)\overline{\psi_j(u)}du\right)^2- \left(\int_{-\lambda}^{\lambda} h_{z_\ell}(u)\overline{\psi_j(u)}du\right)^2\right),
\end{align}
and where the last identity follows from \eqref{Eq:IdentityFUV}. Note that $\alpha_j$ is real by \eqref{Eq:RealIntegral}. 

To prove the inequalities \eqref{Eq:SimpleRealZeros}, \eqref{Eq:DistinctZeros}, \eqref{Eq:SumSimpleCritical} and \eqref{Eq:SimpleOrCritical} we shall establish the corresponding lower bounds for the sum $\sum_{1\leq j\leq D_W} \alpha_j^2$. To this end we split this sum into three parts $1\leq j\leq D_U$, $D_U+1\leq j\leq D_V$ and $D_V+1\leq j\leq D_W$. We assume that these three ranges are all non-empty, otherwise the corresponding sum will just be $0$. 

 We shall first record several important estimates concerning sums of the $\alpha_j$ in these three ranges. Since $f_{x_{\ell}}\in U$ for all $n+1\leq \ell\leq n+r$ and $g_{z_{d}}\in U$ for all $1\leq d \leq k$ then by Parseval's theorem and \eqref{Eq:IntegralFz} we obtain 
\begin{align*} 
\sum_{j=1}^{D_U} \left(\int_{-\lambda}^{\lambda} f_{x_{\ell}}(u) \overline{\psi_j(u)} du\right)^2&= \int_{-\lambda}^{\lambda} |f_{x_{\ell}}(u)|^2 du=1,\\
\textup{ and } \quad \sum_{j=1}^{D_U} \left(\int_{-\lambda}^{\lambda} g_{z_{d}}(u) \overline{\psi_j(u)} du\right)^2&= \int_{-\lambda}^{\lambda} |g_{z_{d}}(u)|^2 du.
\end{align*}
Furthermore, by Bessel's inequality we have  
$$ 
\sum_{j=1}^{D_U} \left(\int_{-\lambda}^{\lambda} h_{z_{\ell}}(u) \overline{\psi_j(u)} du\right)^2\leq  \int_{-\lambda}^{\lambda} |h_{z_{\ell}}(u)|^2 du,
$$
for all $1\leq \ell\leq k$. Combining the above three estimates with the definition of $\alpha_j$ in \eqref{Eq:DefinitionAlpha} together with the identities \eqref{Eq:IntegralFz} and \eqref{Eq:IntegralGH} we get
\begin{align}\label{Eq:FirstSumAlphaj}
\sum_{j=1}^{D_U} \alpha_j & = \sum_{\ell=1}^{n+ r} m_{x_{\ell}}\sum_{j=1}^{D_U} \left(\int_{-\lambda}^{\lambda} f_{x_{\ell}}(u)\overline{\psi_j(u)}du\right)^2\nonumber\\
& \qquad \qquad+ 2\sum_{\ell=1}^k  m_{z_{\ell}}  \bigg(\sum_{j=1}^{D_U} \bigg(\int_{-\lambda}^{\lambda} g_{z_\ell}(u)\overline{\psi_j(u)}du\bigg)^2- \sum_{j=1}^{D_U}\bigg(\int_{-\lambda}^{\lambda} h_{z_\ell}(u)\overline{\psi_j(u)}du\bigg)^2\bigg) \\
& \geq  \sum_{\ell=n+1}^{n+ r} m_{x_{\ell}}+ 2\sum_{\ell=1}^k  m_{z_{\ell}} 
\geq 2 (r+k), \nonumber
\end{align} 
since $m_{x_{\ell}}\geq 2$ for all $n+1\leq \ell \leq n+r$. For the second range $D_U+1\leq j\leq D_V$ we note that $\psi_j$ is orthogonal to $U$ and hence by \eqref{Eq:DefinitionAlpha} we obtain 
\begin{align}\label{Eq:SumAlpha_jSecondRange}\sum_{j=D_U+1}^{D_V} \alpha_j &= \sum_{\ell=1}^{n} \sum_{j=D_U+1}^{D_V}\left(\int_{-\lambda}^{\lambda} f_{x_\ell}(u) \overline{\psi_{j}(u)}du \right)^2- 2\sum_{\ell=1}^k  m_{z_{\ell}} \sum_{j=D_U+1}^{D_V}\bigg(\int_{-\lambda}^{\lambda} h_{z_\ell}(u)\overline{\psi_j(u)}du\bigg)^2 \nonumber
\\&\leq \sum_{\ell=1}^n \int_{-\lambda}^{\lambda}|f_{x_\ell}(u)|^2 du= n,
\end{align}
by Bessel's inequality and \eqref{Eq:IntegralFz}.
Now, in the last range $D_V+1\leq j\leq D_W$, the functions $\psi_j$ are orthogonal to $V$ and hence to all of the $f_{x_{\ell}}$ and the $g_{z_{\ell}}$. Therefore we obtain from \eqref{Eq:DefinitionAlpha} that 
\begin{equation}\label{Eq:NegativeAlphaj}
\alpha_j= -2 \sum_{\ell=1}^k m_{z_\ell} \left(\int_{-\lambda}^{\lambda} h_{z_\ell}(u)\overline{\psi_j(u)}du\right)^2\leq 0.
\end{equation}
Finally, we estimate the total sum $\sum_{1\leq j\leq D_W} \alpha_j$. Using the definition of $\alpha_j$ in \eqref{Eq:DefinitionAlpha} together with Parseval's theorem and the identities  \eqref{Eq:IntegralFz} and \eqref{Eq:IntegralGH} we find that 
\begin{equation}\label{Eq:SumAlphaj}
\sum_{j=1}^{D_W} \alpha_j= \sum_{\ell=1}^{n+ r} m_{x_{\ell}}+ 2\sum_{\ell=1}^{k} m_{z_{\ell}}= \sum_{z\in \mc{Z}}1.
\end{equation}

We start by proving \eqref{Eq:SimpleRealZeros}. 
In the first range we use the basic inequality $a^2+4\geq 4a$, valid for all real numbers $a$, together with \eqref{Eq:FirstSumAlphaj}. This gives 
\begin{equation}\label{Eq:FirstRange0}
\sum_{j=1}^{D_U} \alpha_j^2 \geq  4\sum_{j=1}^{D_U} \alpha_j -4 (k+r) \geq  2\sum_{j=1}^{D_U} \alpha_j,
\end{equation}
since $D_U=k+r$.  In the second range we only use the basic inequality $a^2+1\geq 2a,$ valid for all real numbers $a$, to obtain
\begin{equation}\label{Eq:SecondRange0}
  \sum_{j=D_U+1}^{D_V} \alpha_j^2\geq 2\sum_{j=D_U+1}^{D_V}\alpha_j - n= 2\sum_{j=D_U+1}^{D_V}\alpha_j -\sum_{\substack{z\in \mc{Z}\cap \R\\ m_z=1}}1.
\end{equation}
 Finally, in the third range we use \eqref{Eq:NegativeAlphaj} which trivially gives
\begin{equation}\label{Eq:ThirdRange}
\sum_{j=D_V+1}^{D_W} \alpha_j^2\geq 2 \sum_{j=D_V+1}^{D_W} \alpha_j. 
\end{equation}
Summing the inequalities \eqref{Eq:FirstRange0}, \eqref{Eq:SecondRange0} and \eqref{Eq:ThirdRange} and using \eqref{Eq:OrthogonalInequality} we derive 
\begin{equation}\label{Eq:MasterInequality}
\int_{-\lambda}^{\lambda}\int_{-\lambda}^{\lambda} |F(u, v)|^2 du dv \geq \sum_{j=1}^{D_W}\alpha_j^2\geq 2\sum_{j=1}^{D_W} \alpha_j- \sum_{\substack{z\in \mc{Z}\cap \R\\ m_z=1}}1.
\end{equation}
Inserting  \eqref{Eq:SumAlphaj} in \eqref{Eq:MasterInequality} and combining it with \eqref{Eq:SecondMoment} completes the proof of \eqref{Eq:SimpleRealZeros}. 

We now establish \eqref{Eq:DistinctZeros}. Using our notation above, the number of distinct elements of $\mc{Z}$ is $n+r+2k.$ We will proceed similarly to the proof of \eqref{Eq:SimpleRealZeros}, but we shall change our treatment for the first and second ranges $1\leq j\leq D_U$ and $D_U+1\leq j\leq D_V$. In the first range we only use the first inequality of \eqref{Eq:FirstRange0}. In the second range we use \eqref{Eq:SumAlpha_jSecondRange} and \eqref{Eq:SecondRange0} which imply
\begin{equation}\label{Eq:SecondRangeDistinct}
    \sum_{j=D_U+1}^{D_V} \alpha_j^2\geq 2\sum_{j=D_U+1}^{D_V}\alpha_j -n\geq 4\sum_{j=D_U+1}^{D_V}\alpha_j- 3n.
\end{equation}
Finally, in the last range $D_V+1\leq j\leq D_W$ we use \eqref{Eq:NegativeAlphaj} which implies the trivial inequality
\begin{equation}\label{Eq:ThirdRangeDistinct}
    \sum_{j=D_V+1}^{D_W} \alpha_j^2\geq 4 \sum_{j=D_V+1}^{D_W} \alpha_j.
\end{equation}
Summing the first inequality of \eqref{Eq:FirstRange0} together with the inequalities \eqref{Eq:SecondRangeDistinct} and \eqref{Eq:ThirdRangeDistinct} gives 
$$ \sum_{j=1}^{D_W} \alpha_j^2\geq 4 \sum_{j=1}^{D_W} \alpha_j- 3n-4(k+r).$$
On the other hand by \eqref{Eq:SumAlphaj} we have 
$\sum_{j=1}^{D_W} \alpha_j= \sum_{z\in \mc{Z}} 1\geq n+2r+2k.$ Thus we deduce that 
$$ \sum_{j=1}^{D_W} \alpha_j^2 \geq 4 \sum_{z\in \mc{Z}} 1- 3n-4(k+r) \geq 3\sum_{z\in \mc{Z}} 1 - 2(n+r+k).
$$
Combining this with \eqref{Eq:SecondMoment} and \eqref{Eq:OrthogonalInequality} yields
$$ n+r+2k\geq \frac{3}{2}\sum_{z\in \mc{Z}} 1- \frac12 \sum_{z, s\in \mc{Z}} K(z-s)^2,$$
as desired.

We now prove \eqref{Eq:SumSimpleCritical}. To this end we split the set $\mc{S}$ of the distinct non-real elements of $\mc{Z}$ into two disjoint subsets $\mc{S}_1$ and $\mc{S}_2$, where $\mc{S}_1$ is the set of simple non-real elements of $\mc{Z}$, and $\mc{S}_2$ is the set of those elements which are non-real and have multiplicity $\geq 2$. Put $|\mc{S}_1|=2p$ and $|\mc{S}_2|=2q$, so that $p+q=k$. Since the number of simple elements of $\mc{Z}$ is $n+2p$, and the number of real elements of $\mc{Z}$ counted with multiplicity is $\geq n+2r,$ it suffices to prove the desired lower bound for $2n+2r+2p$.  By \eqref{Eq:FirstSumAlphaj} we have 
\begin{equation}\label{Eq:SumAlphajFirstLarge}
   \sum_{j=1}^{D_U} \alpha_j\geq \sum_{\ell=n+1}^{n+ r} m_{x_{\ell}}+ 2\sum_{\ell=1}^k  m_{z_{\ell}}\geq 2r+2p+4q.  
\end{equation}
Combining this bound with the first inequality of \eqref{Eq:FirstRange0}, together with the inequalities \eqref{Eq:SumAlpha_jSecondRange}, \eqref{Eq:NegativeAlphaj}, and \eqref{Eq:SecondRange0}  we deduce that 
\begin{align}\label{Eq:EquationSimple+Real}
    \sum_{j=1}^{D_W} \alpha_j^2 & \geq \left(4\sum_{j=1}^{D_U}\alpha_j- 4(r+ p+q)\right) + \left(2\sum_{j=D_U+1}^{D_V}\alpha_j -n \right)+3\sum_{j=D_V+1}^{D_W}\alpha_j \nonumber\\
    & \geq 3\sum_{j=1}^{D_W} \alpha_j- (2r+2p+2n). 
\end{align}
Therefore, \eqref{Eq:SumSimpleCritical} follows upon combining this inequality with \eqref{Eq:SecondMoment}, \eqref{Eq:OrthogonalInequality} and \eqref{Eq:SumAlphaj}. 

Finally, we prove \eqref{Eq:SimpleOrCritical}. Here the number of elements of $\mc{Z}$ which are real or simple (or both) is $\geq n+2r+ 2p$, so we are going to seek a lower bound for this quantity. Let $t>2$ be a parameter to be chosen. In the first range we use the basic inequality $a^2\geq 2a t-t^2,$ valid for all $a\in \R$ to get
\begin{equation}\label{Eq:FirstRange3}
\sum_{j=1}^{D_U}\alpha_j^2\geq 2t\sum_{j=1}^{D_U} \alpha_j - t^2(r+p+q).
\end{equation}
In the second and third ranges we use the exact same ingredients as before, namely \eqref{Eq:NegativeAlphaj} and  \eqref{Eq:SecondRange0}. Then similarly to \eqref{Eq:EquationSimple+Real} we get
$$
\sum_{j=1}^{D_W} \alpha_j^2 \geq \left(2 t\sum_{j=1}^{D_U}\alpha_j- t^2(r+ p+q)\right) + \left(2\sum_{j=D_U+1}^{D_V}\alpha_j -n \right)+A\sum_{j=D_V+1}^{D_W}\alpha_j.
$$
By our assumption together with \eqref{Eq:SecondMoment}, \eqref{Eq:OrthogonalInequality} and \eqref{Eq:SumAlphaj} we have $\sum_{j=1}^{D_W} \alpha_j^2\leq A \sum_{j=1}^{D_W}\alpha_j$. Therefore we deduce that 
$$
(2 t- A)\sum_{j=1}^{D_U}\alpha_j- t^2(r+ p+q) + (2-A)\sum_{j=D_U+1}^{D_V}\alpha_j -n\leq 0.
$$
Our goal now is to choose the parameter $t$ in such a way as to balance the weights of the first and second sums of $\alpha_j$ in order to cancel the contribution of $q$ and obtain a multiple of $n+2r+2p$. To this end let $b_1, b_2$ and $b_3$ be positive real numbers such that $b_1+ b_2= 2t-A$ and $b_1-b_3= 2-A$. Then, using \eqref{Eq:SumAlpha_jSecondRange} and \eqref{Eq:SumAlphajFirstLarge} we deduce that 
\begin{equation}\label{Eq:EquilibreSumAlphaj}
b_1 \sum_{j=1}^{D_V} \alpha_j + b_2(2r+2p+4q)-t^2(r+p+q)-(b_3+1) n \leq 0. 
\end{equation}
We now choose $b_1, b_2$ and $b_3$ such that $4b_2=t^2$ and $t^2/2-b_2= b_3+1$ and hence $b_3= t^2/4-1.$ On the other hand we have $b_1+b_2=2t-A$ and $b_1-b_3=2-A$ and hence $b_2+b_3=2t-2.$ This shows that our parameter $t$ satisfies the equation $t^2-4t+2=0$ and hence $t=2+\sqrt{2}$ since $t>2$. With this choice we have $b_1= 5/2+\sqrt{2}-A$, $b_2=3/2+\sqrt{2}$
and $b_3=1/2+\sqrt{2}.$ Inserting these in \eqref{Eq:EquilibreSumAlphaj} and using \eqref{Eq:NegativeAlphaj} and \eqref{Eq:SumAlphaj} we get 
$$ \left(\frac52+\sqrt{2}-A\right)\sum_{z\in \mc{Z}} 1\leq \left(\frac52+\sqrt{2}-A\right)\sum_{j=1}^{D_V} \alpha_j \leq \left(\frac32+\sqrt{2}\right)(n+2r+2p). $$
This completes the proof.
\end{proof}


\section{Proof of Theorem \ref{Thm:Main}}\label{Sec:AnalyticPart}

For complex numbers $z$ with $|\re(z)|<2$ we define
$$
w(z):=\frac{4}{4-z^2}.
$$
We shall use the following unconditional pair correlation formula established by Baluyot, Goldston, Suriajaya and
Turnage-Butterbaugh in \cite{BGST24}. Its proof follows the same argument as Montgomery's original proof under RH \cite{Mo73}, while an unconditional form of the special case corresponding to the Fej\'er kernel was previously obtained by Aryan \cite{Ary22}.

\begin{lem}[Lemma 5 of \cite{BGST24}]\label{Lem:BGST}  Let $f\in L^{1}(\mathbb R)$ be a real-valued even function with support in $[-1,1]$, and suppose that $f$ is Lipschitz continuous at $x =0$ (that is, $
f(x)-f(0)=O(|x|)$ as $x\to 0$). Then we have
\begin{align*}
\sum_{\substack{\rho, \rho' \\ 0<\gamma,\gamma'\le T}} \widehat{f}\left(i(\rho -\rho')\frac{\log T}{2\pi}\right) w(\rho-\rho') = \frac{T}{2\pi}\log T
\left(f(0) + 2\int_{0}^1 \alpha f(\alpha)\,d\alpha+O_f\left(\frac1{\sqrt{\log T}}\right)\right).
\end{align*}
\end{lem}

To combine Proposition~\ref{Pro:FourierHilbert} with the pair correlation formula in Lemma~\ref{Lem:BGST}, we must remove the weight $w(\rho-\rho')$.
Indeed, let $\eta\in C_c^\infty((-1/2,1/2))$ be a real valued even function such that $\widehat{\eta^2}(0)=1$. Define $K:=\widehat{\eta^2}$ and 
$
Q:=\eta^2*\eta^2,
$ where here and throughout we define
$$
(f*g)(x):=\int_{-\infty}^{\infty} f(u)g(x-u)\,du
$$
for all $f, g\in L^2(\R)$.
Then $\widehat Q=K^2$ and $Q$ is supported in $[-1, 1]$. Hence  a direct application of Lemma~\ref{Lem:BGST} gives an asymptotic formula for the weighted sum
$$
\sum_{\substack{\rho,\rho'\\0<\gamma,\gamma'\leq T}}
K\left(
i(\rho-\rho')\frac{\log T}{2\pi}
\right)^2
w(\rho-\rho').
$$
On the other hand, in order to use Proposition~\ref{Pro:FourierHilbert} we need to estimate the corresponding sum without the factor $w(\rho-\rho')$. This factor cannot be removed by directly applying Lemma~\ref{Lem:BGST} to a different test function. Indeed, writing
$
z=i(\rho-\rho')\frac{\log T}{2\pi},
$
we have
$$
w(\rho-\rho')
=
\left(
1+\frac{\pi^2z^2}{(\log T)^2}
\right)^{-1}.
$$
Thus, the identity
$$
K(z)^2=\widehat f(z)w(\rho-\rho')
$$
would require
$$
\widehat f(z)
=
\left(
1+\frac{\pi^2z^2}{(\log T)^2}
\right)K(z)^2,
$$
which depends on $T$. Hence, Lemma~\ref{Lem:BGST}  cannot be applied directly to such a test function without an additional uniformity argument. To overcome this we express the desired kernel as a linear combination of two fixed test functions to which Lemma~\ref{Lem:BGST} can be applied separately.

\begin{lem}\label{Lem:ConstructionKernel}
Put
\begin{equation}\label{Eq:MontgomeryTaylorConstant}
C_{\mathrm{MT}}
:=
\frac12+\frac1{\sqrt2}
\cot\left(\frac1{\sqrt2}\right)
=
1.3274992963205\ldots.    
\end{equation}
Then for every $\varepsilon>0$, there exists a real valued even function
$
\eta_\varepsilon\in C_c^\infty((-1/2,1/2))
$
such that, for the kernel
$
K_\varepsilon(z):=\widehat{\eta_\varepsilon^2}(z),
$
we have
$
K_\varepsilon(0)=1,
$
and
$$
\sum_{\substack{\rho,\rho'\\0<\gamma,\gamma'\leq T}}
K_\varepsilon\left(
i(\rho-\rho')\frac{\log T}{2\pi}
\right)^2
=
\left(
C_{\eta_\varepsilon}
+
O_{\varepsilon}\left(
\frac{1}{\sqrt{\log T}}
\right)
\right)
\frac{T}{2\pi}\log T,
$$
where
$
\left|C_{\eta_\varepsilon}-C_{\mathrm{MT}}\right|
<\varepsilon.
$
\end{lem}
\begin{rem} Before proving the lemma, we briefly compare it with the corresponding
formulation in the Claude proof \cite{AlFu26}. Put
$
L=\log(T/(2\pi)),
\gamma_\rho=-i(\rho-\frac12),$ and $
\psi(u)=\cos(\sqrt2u).$ Choose a smooth cutoff $\chi$ such that $\chi(u)=0$ for $u\leq0$
and $\chi(u)=1$ for $u\geq1$. Define
$\phi(u)=
\chi\left(\frac L2+u\right)
\chi\left(\frac L2-u\right)\psi(u/L)^{1/2},$
and $
\alpha_k=T+2\pi k/L,$ for $
0\leq k<d:=\left\lfloor LT/(2\pi)\right\rfloor.
$
Then the second-moment formula used in \cite{AlFu26} takes the form (see Theorem 5.7 of \cite{AlFu26})
$$
\frac1{a^2L^4}
\sum_{0\leq k,k'<d}
\Bigg(
\sum_{\substack{\rho\\ \re(\gamma_\rho)\in [T-\sqrt T,2T+\sqrt T)}}
\widehat\phi\left(\frac{\gamma_\rho-\alpha_k}{2\pi}\right)
\widehat\phi\left(\frac{\gamma_\rho-\alpha_{k'}}{2\pi}\right)
\Bigg)^2
=
\left(C_{\mathrm{MT}}+o(1)\right)\frac{T}{2\pi} \log T,$$
where 
$a=\frac1L\int_{\mathbb R}\phi(u)^2\,du.$ 
\end{rem}
\begin{proof}[Proof of Lemma \ref{Lem:ConstructionKernel}]
Let $
I=[-1/2,1/2]
$
and define
$$
f_0(x)
:=\begin{cases}
\frac{\cos(\sqrt2x)}
{\sqrt2\sin(1/\sqrt2)} &\textup{ if } x\in I,\\
0 & \textup{ otherwise.}
 \end{cases}
$$
Since
$
|\sqrt2x|\leq\frac1{\sqrt2}<\frac{\pi}{2}$ for all $x\in I$,
the function $f_0$ is strictly positive on $I$. Moreover, $f_0$ is real valued and even, and we have
$$
\int_{-\infty}^{\infty}f_0(x)\,dx
=
\frac{1}{\sqrt2\sin(1/\sqrt2)}
\int_{-1/2}^{1/2}\cos(\sqrt2x)\,dx
=1.
$$
For $0<\delta<1/4$, let
$
\psi_\delta\in C_c^\infty((-1/2,1/2))
$
 be an even function such that
$
0\leq\psi_\delta\leq1
$ and
$
\psi_\delta(x)=1$
for
$|x|\leq\frac12-\delta.$
Put
$$
A_\delta
:=
\int_{-\infty}^{\infty}\psi_\delta(x)^2f_0(x)\,dx.
$$
We now define
$$
\eta_\delta(x)
:=
\frac{\psi_\delta(x)\sqrt{f_0(x)}}{\sqrt{A_\delta}}
\quad \textup{ and } \quad 
f_\delta(x):=\eta_\delta(x)^2
=
\frac{\psi_\delta(x)^2f_0(x)}{A_\delta}.
$$
Since $f_0$ is smooth and strictly positive on $(-\frac12, \frac12)$ it follows that
$
\eta_\delta,f_\delta\in C_c^\infty\big((-1/2,1/2)\big).
$
Moreover, $\eta_\delta$ and $f_\delta$ are real valued and even, $f_\delta\geq0$, and
$$
\int_{-\infty}^{\infty}f_\delta(x)\,dx=1.
$$
Hence, if
$
K_\delta(z):=\widehat{f_\delta}(z)
=\widehat{\eta_\delta^2}(z),
$
then
$
K_\delta(0)=1.
$
We now define
$
Q_\delta:=f_\delta*f_\delta.
$
Since $f_\delta\in C_c^\infty((-1/2,1/2))$, we have
$
Q_\delta\in C_c^\infty((-1,1)).
$
In particular, $Q_\delta$ and $Q_\delta''$ are real-valued even functions supported in $[-1,1]$, and both are Lipschitz continuous at $0$. Thus, both functions satisfy the hypotheses of Lemma~\ref{Lem:BGST}.
Moreover we have $
\widehat Q_\delta(z)
=
\widehat f_\delta(z)^2
=
K_\delta(z)^2
$
for all $z\in \C$. 
We now define
\begin{equation}\label{Eq:IdentityFourierLinear}
r_{\delta,T}
:=
Q_\delta-\frac{Q_\delta''}{4(\log T)^2}.
\end{equation}
Since $Q_\delta$ is smooth and compactly supported, integration by parts twice gives
$$
\widehat{Q_\delta''}(z)
=
(2\pi iz)^2\widehat Q_\delta(z)
=
-4\pi^2z^2K_\delta(z)^2.
$$
Therefore we derive
$$
\widehat r_{\delta,T}(z)
=
\left(
1+\frac{\pi^2z^2}{(\log T)^2}
\right)K_\delta(z)^2,
$$
and hence
$$
\widehat r_{\delta,T}\left(
i(\rho-\rho')\frac{\log T}{2\pi}
\right)
=
\left(
1-\frac{(\rho-\rho')^2}{4}
\right)
K_\delta\left(
i(\rho-\rho')\frac{\log T}{2\pi}
\right)^2.
$$
Thus we deduce that
\begin{align}\label{Eq:LinearCombinationKernel}
& K_\delta\left(
i(\rho-\rho')\frac{\log T}{2\pi}
\right)^2=\widehat r_{\delta,T}\left(
i(\rho-\rho')\frac{\log T}{2\pi}
\right)
w(\rho-\rho')\nonumber\\
& = \widehat{Q_{\delta}}\left(
i(\rho-\rho')\frac{\log T}{2\pi}
\right)
w(\rho-\rho')- \frac{1}{4(\log T)^2}\widehat{Q_{\delta}''}\left(
i(\rho-\rho')\frac{\log T}{2\pi}
\right)
w(\rho-\rho'),
\end{align}
by \eqref{Eq:IdentityFourierLinear} and the linearity of the Fourier transform.
Moreover, by Lemma~\ref{Lem:BGST} we obtain
\begin{align*}
&\sum_{\substack{\rho,\rho'\\0<\gamma,\gamma'\leq T}}\widehat{Q_{\delta}}\left(
i(\rho-\rho')\frac{\log T}{2\pi}
\right)
w(\rho-\rho')\\
& \quad \quad =
\frac{T}{2\pi}\log T
\left(
Q_{\delta}(0) + 2\int_{0}^1 \alpha Q_{\delta}(\alpha)\,d\alpha
+O_\delta\left(\frac{1}{\sqrt{\log T}}\right)
\right),
\end{align*}
and similarly 
\begin{align*}
&\sum_{\substack{\rho,\rho'\\0<\gamma,\gamma'\leq T}}\widehat{Q_{\delta}''}\left(
i(\rho-\rho')\frac{\log T}{2\pi}
\right)
w(\rho-\rho')\\
& \quad \quad =
\frac{T}{2\pi}\log T
\left(
Q_{\delta}''(0) + 2\int_{0}^1 \alpha Q_{\delta}''(\alpha)\,d\alpha
+O_\delta\left(\frac{1}{\sqrt{\log T}}\right)
\right).
\end{align*}
Combining these estimates with \eqref{Eq:LinearCombinationKernel} we derive
$$
\begin{aligned}
\sum_{\substack{\rho,\rho'\\0<\gamma,\gamma'\leq T}}
K_\delta\left(
i(\rho-\rho')\frac{\log T}{2\pi}
\right)^2
&=
\frac{T}{2\pi}\log T
\left(
Q_{\delta}(0) + 2\int_{0}^1 \alpha Q_{\delta}(\alpha)\,d\alpha
+O_\delta\left(\frac{1}{\sqrt{\log T}}\right)
\right).\\
\end{aligned} 
$$
We now write $$C_\delta
:=
Q_{\delta}(0) + 2\int_{0}^1 \alpha Q_{\delta}(\alpha)\,d\alpha,$$ and $Q_0:=f_0*f_0$. Since $f_\delta\to f_0$ in
$L^1(\mathbb R)\cap L^2(\mathbb R)$  we have 
$$
\lim_{\delta\to 0} C_\delta=
Q_0(0)+2\int_0^1\alpha Q_0(\alpha)\,d\alpha= C_{\rm MT},
$$ 
 where the last calculation was performed by Montgomery and Taylor (see \cite{Mo75}).
Thus we may choose $\delta>0$ so small such that
$
|C_\delta-C_{\mathrm{MT}}|<\varepsilon.
$
We fix such a $\delta$, and put
$
\eta_\varepsilon:=\eta_\delta,$
$K_\varepsilon:=K_\delta$ and 
$C_{\eta_\varepsilon}:=C_\delta.$ 
This completes the proof.
\end{proof}

\begin{rem}\label{Rem:MT}
It follows from Corollary 14 of Carneiro, Chandee, Littmann and Milinovich \cite{CCLM} that $C_{\eta_\varepsilon}\geq C_{\rm MT}$. The corresponding extremal function was originally identified by Montgomery and Taylor \cite{Mo75}. 
Together with the preceding construction, this shows that $C_{\mathrm{MT}}$ is the optimal constant that one can obtain using our method.

\end{rem}

\begin{proof}[Proof of Theorem~\ref{Thm:Main}]
Fix $0<\varepsilon<1/2$, and let $\eta_\varepsilon$, $K_\varepsilon$ and
$C_{\eta_\varepsilon}$ be given by Lemma~\ref{Lem:ConstructionKernel}.
For $T\geq 20$, let $\mc Z_T$ be the finite multiset
$$
\mc Z_T
:=
\left\{
i\left(\rho-\frac12\right)\frac{\log T}{2\pi}
:
0<\gamma\leq T
\right\},
$$
where every zero $\rho$ occurs with its multiplicity. The functional
equation of the zeta function implies that $1-\overline{\rho}$
is a zero with the same multiplicity as $\rho$. Moreover,
$$
\overline{
i\left(\rho-\frac12\right)\frac{\log T}{2\pi}
}
=
i\left(1-\overline{\rho}-\frac12\right)
\frac{\log T}{2\pi}.
$$
Thus, $\mc Z_T$ is invariant under complex conjugation. Furthermore,
an element of $\mc Z_T$ is real if and only if the corresponding zero
satisfies $\beta=1/2$, and it is simple if and only if that zero is
simple. Therefore, Proposition~\ref{Pro:FourierHilbert} gives
\begin{align*}
N_0^s(T)
& \geq
2N(T)
-
\sum_{\substack{\rho,\rho'\\0<\gamma,\gamma'\leq T}}
K_\varepsilon\left(
i(\rho-\rho')\frac{\log T}{2\pi}
\right)^2,\\
N_d(T)
&\geq
\frac32N(T)
-
\frac12\sum_{\substack{\rho,\rho'\\0<\gamma,\gamma'\leq T}}
K_\varepsilon\left(
i(\rho-\rho')\frac{\log T}{2\pi}
\right)^2,
\end{align*}
and 
$$N_s(T)+ N_0(T)
\geq
3 N(T)
-
\sum_{\substack{\rho,\rho'\\0<\gamma,\gamma'\leq T}}
K_\varepsilon\left(
i(\rho-\rho')\frac{\log T}{2\pi}
\right)^2.
$$
By Lemma~\ref{Lem:ConstructionKernel}, we have 
\begin{equation}\label{Eq:BestKernelEstimate}
\sum_{\substack{\rho,\rho'\\0<\gamma,\gamma'\leq T}}
K_\varepsilon\left(
i(\rho-\rho')\frac{\log T}{2\pi}
\right)^2=\left(
C_{\eta_\varepsilon}
+
O_{\varepsilon}\left(
\frac1{\sqrt{\log T}}
\right)
\right)
\frac{T}{2\pi}\log T.
\end{equation}
Hence, we deduce that
$$
\frac{N_0^s(T)}{N(T)}
\geq
2-C_{\eta_\varepsilon}+o_{\varepsilon}(1), \quad \frac{N_d(T)}{N(T)}
\geq
\frac32-\frac{C_{\eta_\varepsilon}}{2}+o_{\varepsilon}(1),
$$
and 
$$ \frac{N_s(T)+N_0(T)}{N(T)} \geq 3-C_{\eta_\varepsilon}+o_{\varepsilon}(1).$$
Moreover, by \eqref{Eq:BestKernelEstimate} and \eqref{Eq:SimpleOrCritical} we get 
$$ \frac{N_{s\cup 0}(T)}{N(T)}\geq \frac{5+2\sqrt{2}-2C_{\eta_{\varepsilon}}}{3+2\sqrt{2}} +o_{\varepsilon}(1).$$
Since
$
\left|C_{\eta_\varepsilon}-C_{\mathrm{MT}}\right|
<\varepsilon$, $C_0= 2-C_{\mathrm{MT}}$, $C_1=3/2-C_{\mathrm{MT}}/2$ and $C_2=(5+2\sqrt{2}-2C_{\mathrm{MT}})/(3+2\sqrt{2})$,
we obtain
$$
\liminf_{T\to\infty}
\frac{N_0^s(T)}{N(T)}
\geq
C_0-\varepsilon,  \quad \liminf_{T\to\infty}
\frac{N_d(T)}{N(T)}
\geq
C_1-\varepsilon,
$$
and 
$$
\liminf_{T\to\infty}
\frac{N_s(T)+N_0(T)}{2N(T)}
\geq
C_1-\varepsilon,  \quad \textup{ and } \quad  \liminf_{T\to\infty}
\frac{N_{s\cup 0}(T)}{N(T)}
\geq
C_2-\varepsilon.
$$
Finally, letting $\varepsilon\to0$ completes the proof.
\end{proof}

\appendix
\section{Formal Certificate by AxiomProver}

AxiomProver, an AI system currently under development by Axiom Math, autonomously generated (from natural language) an unconditional formal certificate for Proposition~\ref{Pro:FourierHilbert}.  Using Proposition~\ref{Pro:FourierHilbert}, a formal certificate was also autonomously generated for Theorem~\ref{Thm:Main} under the assumptions of \cite[Lemma~5]{BGST24} (see Lemma~\ref{Lem:BGST} above) and the standard Riemann--von Mangoldt asymptotic $N(T)\sim \frac{1}{2\pi}T\log T$ (see \cite[Theorem~9.4]{Tit86}).  The certificates can be found in
\begin{center}
\url{https://github.com/AxiomMath/ZetaZerosV2}
\end{center}
This repository contains a formal challenge file containing the statements of the three main results, which can be mechanically verified using the Comparator tool in Lean.

\end{document}